\documentclass[11pt,a4paper]{article}

\usepackage[utf8]{inputenc}
\usepackage[T1]{fontenc}
\usepackage[colorlinks]{hyperref}
\usepackage{mathtools, amsthm, amsfonts, amssymb}
\usepackage{accents} 
\usepackage{microtype}

\hypersetup{
  linkcolor=[rgb]{0.3,0.3,0.6},
  citecolor=[rgb]{0.2, 0.6, 0.2},
  urlcolor=[rgb]{0.6, 0.2, 0.2}
}

\usepackage{thmtools}
\usepackage{mathrsfs}
\usepackage{textcomp}
\usepackage{graphicx}
\usepackage{nicefrac}
\usepackage{enumerate}
\usepackage{wasysym}
\usepackage[normalem]{ulem}
\usepackage{upgreek}

\usepackage{paralist}
\usepackage{xcolor}
\usepackage{etoolbox}
\usepackage{mathdots}
\usepackage{wrapfig}
\usepackage{floatflt}
\usepackage{tensor} 
\usepackage{lscape}
\usepackage{stmaryrd}
\usepackage[enableskew]{youngtab}
\usepackage{ytableau}
\usepackage{pifont}

\usepackage{tikz}
\usetikzlibrary{arrows}
\usetikzlibrary{matrix}
\usetikzlibrary{positioning}
\usetikzlibrary{calc}

\usepackage{xcolor}
\definecolor{red}{rgb}{1,0,0}

\newcommand{\vvirg}{ , \dots , }

\newcommand{\bfm}{\mathbf{m}}

\newcommand{\calA}{\mathcal{A}}

\newcommand{\bbF}{\mathbb{F}}

\newcommand{\bbK}{\mathbb{K}}

\newcommand{\bbN}{\mathbb{N}}

\renewcommand{\phi}{\varphi}
\newcommand{\eps}{\varepsilon}

\newcommand{\rank}{\mathrm{rank}}

\DeclareMathOperator{\subrank}{Q}
\DeclareMathOperator{\tensorrank}{R}
\DeclareMathAccent{\wtilde}{\mathord}{largesymbols}{"65}
\DeclareMathOperator{\aQ}{\underaccent{\wtilde}{Q}}
\DeclareMathOperator{\asympsubrank}{\underaccent{\wtilde}{Q}}

\newcommand{\GL}{\mathrm{GL}}

\usepackage[top=2.5cm,bottom=2.5cm,left=3cm,right=3cm]{geometry}

\usepackage{booktabs}

\newcommand{\w}{\mathsf{W}}
\newcommand{\un}{\mathsf{I}}
\newcommand{\n}{\mathsf{N}}
\newcommand{\deter}{\mathsf{D}}

\DeclareMathOperator{\maxrank}{max-rank}

\usepackage{thmtools, thm-restate} 
\declaretheorem[name=Theorem, parent=section]{theorem}

\declaretheorem[name=Lemma, sibling=theorem]{lemma}

\theoremstyle{definition}

\theoremstyle{remark}

\numberwithin{equation}{section}

\newcommand{\degengeq}{\unrhd}

\begin{document}
\mbox{}\vspace{1em}
\begin{center}
	{\LARGE The next gap in the subrank of 3-tensors}
\\[1cm] \large

\setlength\tabcolsep{0em}
\newcommand{\myPad}{\hspace{2em}}
\centerline{%
\begin{tabular}{c@{\myPad}c}
	Fulvio Gesmundo & Jeroen Zuiddam\\[0.2em]
    
    \textsf{fgesmund@math.univ-toulouse.fr} & \textsf{j.zuiddam@uva.nl}\\[0.2em]
     Institut de Mathématiques de Toulouse  & University of Amsterdam
\end{tabular}%
}

\vspace{9mm}

\large
{\today}

\vspace{9mm}
\bf Abstract
\end{center}
\normalsize
\noindent
Recent works of Costa--Dalai, Christandl--Gesmundo--Zuiddam, Blatter--Draisma--Rupniew\-ski, and Briët--Christandl--Leigh--Shpilka--Zuiddam have investigated notions of discreteness and gaps in the possible values that asymptotic tensor ranks can take. In particular, it was shown that the asymptotic subrank and asymptotic slice rank of any nonzero 3-tensor is equal to 1, equal to 1.88, or at least 2 (over any field), and that the set of possible values of these parameters is discrete (in several regimes). We determine exactly the next gap, showing that the asymptotic subrank and asymptotic slice rank of any nonzero 3-tensor is equal to 1, equal to 1.88, equal to 2, or at least 2.68.

\vspace{1em}
2020 Math.~Subj.~Class.: (primary) 15A69, (secondary) 14N07, 15A72, 68R05

\section{Introduction} \label{intro}

Unlike matrix rank, many natural notions of tensor rank (in particular, those defined in an asymptotic or amortized manner) may take non-integral values. For instance, the asymptotic subrank and asymptotic slice rank of the tensor $e_1 \otimes e_2 \otimes e_2 + e_2 \otimes e_1 \otimes e_2 + e_2 \otimes e_2 \otimes e_1$ equals $2^{h(1/3)} \approx 1.88$ where~$h$ is the binary entropy function. Applications often ask for determining the value of these parameters for specific tensors. This raises the question: What values can such parameters take? Are there gaps between the values? Are there accumulation points?

\paragraph{Concrete gaps.}
We briefly summarize the known results in this area, which can roughly be grouped into results about ``concrete gaps'' and the ``general structure'' of gaps.
Since the first gap result of Strassen \cite[Lemma~3.7]{strassen1988asymptotic}, who proved that the asymptotic subrank (and as a consequence, the asymptotic partition rank) of any nonzero $k$-tensor is either 1 or at least~$2^{2/k}$, several works have investigated notions of discreteness and gaps in the values of tensor parameters. Costa and Dalai \cite{DBLP:journals/jcta/CostaD21} proved that the asymptotic slice rank of any nonzero $k$-tensor is either $1$ or at least $2^{h(1/k)}$ where $h$ is the binary entropy function.
Christandl, Gesmundo and Zuiddam~\cite{cgz} proved the stronger analogous statement for the asymptotic subrank, and moreover showed for $k=3$ that the asymptotic subrank is equal to 1, equal to $2^{h(1/3)}\approx 1.88$, or at least $2$, leaving as an open problem whether the set of all possible values is discrete, and in particular what is the next possible value.

\paragraph{General structure.}
Blatter, Draisma and Rupniewski \cite{bdr} proved that the set of values of every normalized monotone tensor parameters over finite fields is well-ordered: in particular, it has no accumulation points from above. 
Christandl, Vrana and Zuiddam \cite{MR4495838} proved, using methods from representation theory and quantum information, that the asymptotic slice rank over the complex numbers takes only finitely many values on tensors of any fixed format, and thus only countably many values in general. Blatter, Draisma and Rupniewski \cite{https://doi.org/10.48550/arxiv.2212.12219} proved 
that a class of asymptotic tensor parameters over complex numbers take only countably many values; this class includes asymptotic subrank and asymptotic slice rank, over arbitrary fields.
Briët, Christandl, Leigh, Shpilka and Zuiddam~\cite{briet2023discreteness} proved that for a general class of asymptotic tensor parameters over several regimes, the set of values of any function in the class is discrete. This includes the asymptotic subrank over finite fields and the asymptotic slice rank over complex numbers.

\paragraph{New results and methods.}
In this paper, we prove a new concrete gap for the asymptotic subrank of 3-tensors over any field, showing that the asymptotic subrank of any nonzero tensor is equal to 1, equal to $2^{h(1/3)}\approx 1.88$,  equal to 2, or at least ${\approx}\, 2.68$ (\autoref{th:new-values}). The last value is the asymptotic subrank of the multiplication tensor of the trivial unital algebra of dimension~$3$. 

To obtain this result, we prove a structural result about restrictions between tensors. In our proof we make use of the notion of the maximum rank in the slice span of a tensor that was also central in \cite{briet2023discreteness}, and in particular we prove that this parameter remains as large as possible under generic restriction. We moreover use a result about degenerating to the trivial algebra of Blaser--Lysikov \cite{DBLP:conf/mfcs/BlaserL16} and a classification of matrix subspaces of low-rank of Atkinson \cite{MR695915} and Eisenbud--Harris \cite{MR954659}.

\section{Gaps in the asymptotic subrank}
In this section we will provide some preliminary definitions, briefly summarize the known results from \cite{cgz}, and discuss the new results in detail.

\subsection{Basic definitions}

We provide basic definitions here. For more background we refer to \cite{strassen1988asymptotic, burgisser1997algebraic,wigderson2022asymptotic}.
Let $\bbF$ be any field. In this work, we consider tensors of order three: let $T \in \bbF^{n_1} \otimes \bbF^{n_2} \otimes \bbF^{n_3}$ be a tensor over $\bbF$ with dimensions $(n_1, n_2, n_3)$. The subrank of $T$, denoted by $\subrank(T)$, is the largest number~$r\in \bbN$ such that there are linear maps $A_i : \bbF^{n_i} \to \bbF^r$ such that $(A_1 \otimes A_2\otimes A_3)T = \sum_{i=1}^r e_i \otimes e_i \otimes e_i$. The asymptotic subrank of $T$ is defined as 
\[
\asympsubrank(T) = \lim_{n \to \infty} \subrank(T^{\boxtimes n})^{1/n}
\]
where $\boxtimes$ is the Kronecker product on tensors. This limit exists and equals the supremum by Fekete's Lemma, since $\subrank$ is super-multiplicative. 

The flattenings of $T$ are the elements in $(\bbF^{n_1} \otimes \bbF^{n_2}) \otimes \bbF^{n_3}$, $\bbF^{n_1} \otimes (\bbF^{n_2} \otimes \bbF^{n_3})$ and $\bbF^{n_2} \otimes (\bbF^{n_1} \otimes \bbF^{n_3})$, obtained by naturally grouping the tensor factors of $T$. They are regarded as matrices, that is tensors of order two.

For any two tensors $T \in \bbF^{n_1} \otimes \bbF^{n_2} \otimes \bbF^{n_3}$ and $S \in \bbF^{m_1} \otimes \bbF^{m_2} \otimes \bbF^{m_3}$ we say $T$ restricts to~$S$ and write $T \geq S$ if there are linear maps $A_i : \bbF^{n_i} \to \bbF^{m_i}$ such that $(A_1 \otimes A_2\otimes A_3)T = S$. In particular, $\subrank(T)$ is the largest number $r\in \bbN$ such that $T \geq \sum_{i=1}^r e_i \otimes e_i \otimes e_i$.

\subsection{Previously known gaps}
Let $\bbK$ be any algebraically closed field.
Let the tensors $\un \in \bbK^2 \otimes \bbK^2 \otimes \bbK^2$ and $\w \in \bbK^2 \otimes \bbK^2 \otimes \bbK^2$ be defined by
\begin{align*}
\un &= e_1 \otimes e_1 \otimes e_1 + e_2 \otimes e_2 \otimes e_2,\\
\w &= e_1 \otimes e_1 \otimes e_2 + e_1 \otimes e_2 \otimes e_1 + e_2 \otimes e_1 \otimes e_1.
\end{align*}
In \cite{cgz} the following classification in terms of $\un$ and $\w$ is proven:
\begin{theorem}[\cite{cgz}]\label{thm:prev}
Let $n_1, n_2, n_3 \in \bbN$ be arbitrary.
For every nonzero $T \in \bbK^{n_1} \otimes \bbK^{n_2} \otimes \bbK^{n_3}$ exactly one of the following is true:
\begin{enumerate}[\upshape(a)]
    \item $T$ has a flattening of rank one;
    \item $\w \geq T$ and $T \geq \w$;
    \item $T \geq \un$.
\end{enumerate}
\end{theorem}
By monotonicity of asymptotic subrank, using the known values $\asympsubrank(\w) = 2^{h(1/3)}$ and $\asympsubrank(\un) = 2$, the following gap theorem for asymptotic subrank can be deduced immediately from \autoref{thm:prev}:
\begin{theorem}[\cite{cgz}]\label{thm:prev-Qtilde}
        For every nonzero $T \in \bbK^{n_1} \otimes \bbK^{n_2} \otimes \bbK^{n_3}$, exactly one of the following is true:
    \begin{enumerate}[\upshape(a)]
    \item $\asympsubrank(T) = 1$;
    \item $\asympsubrank(T) = c_1 \coloneqq 2^{h(1/3)} \approx 1.88988$;
    \item $\asympsubrank(T) \geq 2$.
    \end{enumerate}
\end{theorem}
\autoref{thm:prev-Qtilde} is also true with $\asympsubrank$ replaced by asymptotic slice rank, since for $\un$ and $\w$ the asymptotic subrank equals the asymptotic slice rank.

\subsection{New result}

We prove a classification that extends \autoref{thm:prev}. To state it we need to define the following tensors.
Let $\n^{(1)}, \n^{(2)},\n^{(3)} \in \bbK^3 \otimes \bbK^3 \otimes \bbK^3$ be defined by
\begin{align*}
\n^{(1)} &= e_1 \otimes e_1 \otimes e_1 + e_2 \otimes e_1 \otimes e_2 + e_2 \otimes e_2 \otimes e_1 + e_3 \otimes e_1 \otimes e_3 + e_3 \otimes e_3 \otimes e_1 \\
\n^{(2)} &= e_1 \otimes e_1 \otimes e_1 + e_2 \otimes e_2 \otimes e_1 + e_1 \otimes e_2 \otimes e_2 + e_3 \otimes e_3 \otimes e_1 + e_1 \otimes e_3 \otimes e_3 \\
\n^{(3)} &= e_1 \otimes e_1 \otimes e_1 + e_1 \otimes e_2 \otimes e_2 + e_2 \otimes e_1 \otimes e_2 + e_1 \otimes e_3 \otimes e_3 + e_3 \otimes e_1 \otimes e_3.
\end{align*}
Note that these tensors are cyclic shifts of each other. As a bilinear map, regarding the first and second tensor factors as ``inputs'' and the third factor as ``output'', $\n^{(3)}$ is the tensor encoding the multiplication map of the $3$-dimensional trivial unital algebra $\bbK[x,y]/(x^2,xy,y^2)$.

Let $\deter \in \bbK^3 \otimes \bbK^3 \otimes \bbK^3$ be defined by
\[
\deter = e_1 \otimes e_2 \otimes e_3 - e_1 \otimes e_3 \otimes e_2 + e_2 \otimes e_1 \otimes e_3 - e_2 \otimes e_3 \otimes e_1 + e_3 \otimes e_1 \otimes e_2 - e_3 \otimes e_2 \otimes e_1.
\]
In other words, $\deter$ is the unique up to rescaling fully skew-symmetric tensor $e_1 \wedge e_2 \wedge e_3 \in \Lambda^3 \bbK^3$. As a trilinear map, it takes three vectors in $\bbK^3$ to the determinant of the $3 \times 3$ matrix whose columns are the three vectors.

For any two tensors $T$ and $S$ we say $T$ degenerates to $S$, and write $T \degengeq S$, if there are linear maps $A(\eps), B(\eps), C(\eps)$ whose coefficients are Laurent polynomials in the formal variable~$\eps$, such that
\[
(A(\eps) \otimes B(\eps) \otimes C(\eps))  T = S + \eps S_1 + \eps^2 S_2 + \cdots + \eps^t S_t
\]
for some arbitrary tensors $S_1, \ldots, S_t$. It is known that asymptotic subrank is monotone under degeneration \cite[Proposition 5.10]{strassen1988asymptotic}, \cite[Proposition 15.26]{burgisser1997algebraic}.

\begin{theorem}\label{thm:new}
Let $n_1, n_2, n_3 \in \bbN$ be arbitrary.
For every nonzero $T \in \bbK^{n_1} \otimes \bbK^{n_2} \otimes \bbK^{n_3}$ exactly one of the following is true:
\begin{enumerate}[\upshape(a)]
    \item $T$ has a flattening of rank one;
    \item $\w \geq T$ and $T \geq \w$;
    \item $T \geq \un$ and $T$ has a flattening of rank two;
    \item all flattenings of $T$ have rank at least three, in which case at least one of the following is true
    \begin{itemize}
        \item $T \degengeq \n^{(i)}$ for some $i \in [3]$;
        \item $T \geq \deter$ and $\deter \geq T$.
    \end{itemize}
\end{enumerate}
\end{theorem}

In order to discuss the resulting gaps in the values of the asymptotic subrank, we need to know the values of $\asympsubrank(\n^{(i)})$ and $\asympsubrank(\deter)$.
Define the number $c_2 \coloneqq 2^{\tau + h(\tau)}$ where $\tau \in (0,1/2)$ is the unique solution of 
    \[
    h(2\tau)- h(\tau)+\tau = 0,
    \]
(where $h$ is the binary entropy function) which numerically evaluates to $c_2 \approx 2.68664$.

The value of $\asympsubrank(\n^{(i)})$ was computed in \cite{strassen1991degeneration}, the one of $\asympsubrank(\deter)$ was computed implicitly in \cite{CopperWinog:MatrixMultiplicationArithmeticProgressions} and more explicitly in \cite{ConGesLanVen:RankBordRankKronPowers}. We record here the two results:
\begin{lemma}
    For every $i \in [3]$, we have $\asympsubrank(\n^{(i)}) = c_2 \approx 2.68664$.
\end{lemma}
\begin{proof}
    The tensor $\n^{(3)}$ is the structure tensor of the 3-dimensional null-algebra as described in \cite[page 168]{strassen1991degeneration}.
    From Equation 6.19 (with $q=2$), we find that $\asympsubrank(\n^{(3)}) = c_2$.  
    The values of $\asympsubrank(\n^{(1)})$ and $\asympsubrank(\n^{(2)})$ are the same as $\asympsubrank(\n^{(3)})$, since they are obtained from $\n^{(3)}$ by permuting the tensor factors.
\end{proof}

\begin{lemma}\label{lem:QD}
    $\asympsubrank(\deter) = 3$.
\end{lemma}
\begin{proof}
    The value of $\asympsubrank(\deter)$ is at most 3 because the flattenings ranks of $\deter$ are 3. The lower bound follows from a standard application of the support functional method of Strassen~\cite[Proposition 5.4]{strassen1991degeneration}. For this we first observe that the support of $\deter$ is tight in the basis that we presented it. The support is symmetric so the minimum over $\theta$ is attained for the uniform $\theta$. The maximum over the probability distributions on the support is attained for the uniform distribution, which gives the required value.
\end{proof}

The essential information from \autoref{lem:QD} is that $\asympsubrank(\deter) \geq \asympsubrank(\n^{(i)})$. The above leads to the following:

\begin{theorem}\label{th:new-values}
    For every nonzero $T \in \bbK^{n_1} \otimes \bbK^{n_2} \otimes \bbK^{n_3}$, exactly one of the following is true:
    \begin{enumerate}[\upshape(a)]
    \item $\asympsubrank(T) = 1$;
    \item $\asympsubrank(T) = c_1 \approx 1.88988$;
    \item $\asympsubrank(T) = 2$;
    \item $\asympsubrank(T) \geq c_2 \approx 2.68664$.
    \end{enumerate}
\end{theorem}
\begin{proof}
    We follow \autoref{thm:new} case by case. In case (a), the tensor has a flattening of rank one (but is nonzero), and thus $\asympsubrank(T) = 1$. In case (b), $T$ is equivalent to $\w$ so $\asympsubrank(T) = \asympsubrank(\w) = c_1$. In case (c), $T \geq \un$ so $\asympsubrank(T) \geq 2$ and $T$ has a flattening of rank two, so $\asympsubrank(T) \leq 2$. In case (d), either $T \degengeq \n^{(i)}$ for some $i \in [3]$, in which case $\asympsubrank(T) \geq \asympsubrank(\n^{(i)}) = c_2$ or $T \degengeq \deter$, in which case $\asympsubrank(T) \geq \asympsubrank(\deter) = 3 \geq c_2$.
\end{proof}
\autoref{th:new-values} is also true with $\asympsubrank$ replaced by asymptotic slice rank, since (by standard results) for each of the tensors~$\un$, $\w$, $\deter$ and $\n^{(i)}$ the asymptotic subrank equals the asymptotic slice rank.
\section{Preliminary Results}

Before proving our main result we discuss three preliminary lemmas that will play a central role in the proof.

\subsection{Rank of slice-spans under restriction}

For any matrix subspace $\mathcal{A}$, let $\maxrank(\mathcal{A})$ be the largest matrix rank of any element of~$\mathcal{A}$. By semicontinuity, $\maxrank(\calA)$ is the rank of generic elements of $\calA$. 

Let $T \in \bbK^{n_1*} \otimes \bbK^{n_2} \otimes \bbK^{n_3}$. For every $i$, let $T^{(i)} : \bbK^{n_i} \to \bbK^{n_j} \otimes \bbK^{n_j}$ be the $i$-th flattening of $T$, with $\{i,j,k\} = \{ 1,2,3\}$. Define $\subrank_i(T) = \maxrank( T^{(i)} (\bbK^{n_i *}))$, where $ T^{(i)} (\bbK^{n_i *})$ is regarded as a linear space of $n_j \times n_k$ matrices. The parameters $\subrank_i$ and their properties were used in \cite{briet2023discreteness} to prove discreteness of asymptotic tensor ranks.

Clearly $\subrank_i(T) \leq \min\{n_j,n_k\}$ for every distinct $i,j,k \in [3]$, and that~$\subrank_i$ is monotone under restriction. We prove that $\subrank_i(T)$ remain \emph{as large as possible} under (generic) restriction. This result, as well as its proof, is similar to \cite[Proposition 3.1]{cgz}.

\begin{lemma}\label{lem:restr-Qi}
    Let $T \in \bbK^{n_1} \otimes \bbK^{n_2} \otimes \bbK^{n_3}$. For any $m_1, m_2, m_3 \in \bbN$ there is a tensor $S \in \bbK^{m_1} \otimes \bbK^{m_2} \otimes \bbK^{m_3}$ such that $T \geq S$ and for every distinct $i,j,k \in [3]$ we have
    \[
    \subrank_i(S) = \min \{ \subrank_i(T), m_j, m_k\}.
    \]
\end{lemma}
\begin{proof}
It is clear that $\subrank_i(S) \leq \min \{ \subrank_i(T), m_j, m_k\}$ for every choice of $i,j,k$. 

For every $i \in [3]$, let $A^{(i)}_1, \ldots, A^{(i)}_{n_i}$ be the $i$-slices of $T$, that is $A^{(i)}_j = T^{(i)}(e_j)$ for a fixed basis $e_1 \vvirg e_{n_i}$ of $\bbK^{n_i*}$. By definition, $T^{(i)}(\alpha) \in \langle A^{(i)}_1 \vvirg A^{(i)}_{n_i} \rangle$ for every $\alpha \in \bbK^{n_i*}$. We have $\rank(T^{(i)}(\alpha)) \leq \subrank_i(T)$ for every $\alpha$ and by semicontinuity of matrix rank equality holds on a Zariski open subset: its complement is given by the vanishing of size $(\subrank_i(T) + 1)$ minors of $T^{(i)}(\alpha)$ regarded as a matrix depending on (the coefficients of) $\alpha$. 

As a consequence, there is a non-empty (Zariski) open set of $\GL_{n_i}$ such that the matrices $B^{(i)}_1, \ldots, B^{(i)}_{n_i}$ obtained by taking any such linear combinations of $A^{(i)}_1, \ldots, A^{(i)}_{n_i}$ has the property that $B^{(i)}_1$ has rank equal to $\subrank_i(T)$: to see this it suffices to consider elements of $\GL_{n_i}$ whose first row, regarded as an element $\alpha \in \bbK^{n_i}$, satisfies $\rank(T^{(i)}(\alpha)) = \subrank_i(T)$.

Similarly, for every $i\in [3]$, there is a non-empty open set of column operations and row operations on $B^{(i)}_1$ such that the new matrix $C^{(i)}_1$ has the property that its submatrix $C^{(i)}_1|_{m_j \times m_k}$ has rank $\min\{\subrank_i(T), m_j, m_k\}$.

The intersection of finitely many non-empty Zariski open subset is Zariski open and dense. Hence, we can obtain the above properties simultaneously for every $i \in [3]$ by acting on $T$ with an operation from the intersection. After these operations, let $S$ be the subtensor obtained by projecting $T$ on the coordinates $[m_1] \times [m_2] \times [m_3]$. The tensor $S$ satisfies $\subrank_i(S) =  \min \{ \subrank_i(T), m_j, m_k\}$.
\end{proof}

\subsection{Degenerating to the trivial algebra}

Recall the tensor $\n^{(i)} \in \bbK^{3} \otimes \bbK^{3} \otimes \bbK^{3}$ for $i \in [3]$. We consider the analog tensor in higher dimension.
For any $n \in \bbN$, let $\n^{(3)}_n \in \bbK^n \otimes \bbK^n \otimes \bbK^n$ be defined by
\begin{align*}
\n^{(3)}_n &= e_1 \otimes e_1 \otimes e_1 + \sum_{i=2}^n (e_1 \otimes e_i \otimes e_i + e_i \otimes e_1 \otimes e_i)
\end{align*}
and let $\n^{(1)}_n, \n^{(2)}_n \in \bbK^n \otimes \bbK^n \otimes \bbK^n$ be obtained from $\n^{(3)}_n$ by cyclically shifting the tensor factors. Note that $\n^{(i)} = \n^{(i)}_3$. As a bilinear map, regarding the first and second tensor factors as ``inputs'' and the third factor as ``output'', $\n^{(3)}_n$ is the tensor encoding the multiplication map of the $n$-dimensional trivial unital algebra $\bbK[x_1 \vvirg x_{n-1}]/\bfm^2$ where $\bfm = (x_1 \vvirg x_{n-1})$ is the ideal generated by the variables.

We will be using the following result, which says that if the slice spans have maximum rank in two directions $i\neq j$, then the tensor degenerates to $\n_n^{(k)}$ in the third direction $k \neq i, k \neq j$.

\begin{lemma}\label{lem:binding-to-null}
    Let $T \in \bbK^{n} \otimes \bbK^{n} \otimes \bbK^{n}$. For every distinct $i,j,k \in [3]$, if $\subrank_i(T) = \subrank_j(T) = n$, then $T \degengeq \n^{(k)}_n$.
\end{lemma}

\autoref{lem:binding-to-null} follows directly from combining the following two lemmas.

\begin{lemma}[Bl\"aser--Lysikov {\cite[Lemma 3.5]{DBLP:conf/mfcs/BlaserL16}}]
    Let $T \in \bbK^{n} \otimes \bbK^{n} \otimes \bbK^{n}$.
    If $\subrank_1(T) = \subrank_2(T) = n$, then $T$ is isomorphic to the multiplication tensor of an $n$-dimensional unital algebra. In particular, 
    \[
    T \simeq \n^{(3)}_n + \sum_{a=2}^n \sum_{b=2}^n \sum_{c=1}^n T_{a,b,c}\, e_a \otimes e_b \otimes e_c
    \]
    for some $T_{a,b,c} \in \bbK$.
\end{lemma}

\begin{lemma}
    The multiplication tensor of an $n$-dimensional unital algebra degenerates to $\n^{(3)}_n$. More precisely, for every $n \in \bbN$ and every $T_{a,b,c} \in \bbK$,
    \[
    \n^{(3)}_n + \sum_{a=2}^n \sum_{b=2}^n \sum_{c=1}^n T_{a,b,c}\, e_a \otimes e_b \otimes e_c \; \degengeq \; \n^{(3)}_n.
    \]
\end{lemma}

\begin{proof}
Let $A(\eps) = B(\eps)$ map $e_1 \mapsto \eps^{-2} e_1$ and $e_j \mapsto \eps e_j$ for $j \geq 2$.
Let $C(\eps)$ map $e_1 \mapsto \eps^{4} e_1$ and $e_j \mapsto \eps e_j$ for $j \geq 2$.
Then 
\[
(A(\eps) \otimes B(\eps) \otimes C(\eps)) \Bigl(\n^{(3)}_n + \sum_{a=2}^n \sum_{b=2}^n \sum_{c=1}^n S_{a,b,c}\, e_a \otimes e_b \otimes e_c\Bigr) = \n^{(3)}_n + \eps^3 U_1 + \eps^6 U_2
\]
for some tensors $U_1, U_2$.
\end{proof}

\subsection{Classification of matrix subspaces of small rank}

The final ingredient is a classification of matrix subspaces of small rank. Let $\mathcal{A}\subseteq \bbK^{n_1\times n_2}$ and $\mathcal{B} \subseteq \bbK^{m_1 \times m_2}$ be matrix subspaces. We say $\mathcal{A}$ and $\mathcal{B}$ are \emph{equivalent} if $\mathcal{A}$ can be obtained from $\mathcal{B}$ by simultaneous invertible row and column operations and adding or removing any number of zero rows or columns. The classification of linear subspaces of small rank is the subject of a long line of research in linear algebra and algebraic geometry. In this section, we only need the classification for rank $2$, which dates back to \cite{AtkLlo:LargeSpaces,MR954659}. In \cite{MR695915}, the classification for rank at most three was obtained, whereas the classification for rank at most four was only recently obtained in \cite{HuLan:LinearSpacesMatrices}.

The classification uses the space of $3 \times 3$ skew-symmetric matrices, which is the space
\[
\left\{\begin{pmatrix}
    0 & a & b\\
    -a & 0 & c\\
    -b & -c & 0
\end{pmatrix} : a,b,c \in \bbK \right\}
\]
Note that the slice span, in each of the three directions, of the tensor $\deter$ defined earlier is equal to the space of $3 \times 3$ skew-symmetric matrices.

\begin{lemma}[Atkinson \cite{MR695915}, Eisenbud--Harris \cite{MR954659}]\label{lem:classification}
Let $\mathcal{A}$ be a matrix subspace.
\begin{itemize}
    \item If $\rank(\mathcal{A}) = 1$, then up to equivalence, $\mathcal{A}$ is supported in a single row or column.
    \item If $\rank(\mathcal{A}) = 2$, then up to equivalence, $\mathcal{A}$ is supported in two rows, or two columns, or a row and a column, or $\mathcal{A}$ is equivalent to the space of $3\times 3$ skew-symmetric matrices.
\end{itemize}
\end{lemma}

\section{Proof of \autoref{thm:new}}
We now give the proof of \autoref{thm:new}.
Let $T \in \bbK^{n_1} \otimes \bbK^{n_2} \otimes \bbK^{n_3}$ be nonzero. Let $r_i = \tensorrank_i(T)$ and $q_i = \subrank_i(T)$ for $i \in [3]$.

\begin{lemma}\label{lem:q-lower-bound}
    If $r_1, r_2, r_3 \geq 3$, then $q_1, q_2, q_3 \geq 2$.
\end{lemma}
\begin{proof}
    If $q_i = 1$ for some $i \in [3]$, then the the span of the $i$-slices is supported on a single row or column by \autoref{lem:classification}. This implies that $r_j = 1$ for some $j \in [3]$.
\end{proof}

We consider two cases:
\begin{itemize}
    \item $q_i, q_j \geq 3$ for some distinct $i,j \in [3]$
    \item $q_i,q_j = 2$ for some distinct $i,j \in [3]$
\end{itemize}

\begin{lemma}\label{lem:N}
    Suppose $r_1, r_2, r_3 \geq 3$. Let $i,j,k \in [3]$ be distinct. If $q_i, q_j \geq 3$, then  $T \degengeq \n^{(k)}$.
\end{lemma}
\begin{proof}
    Suppose $q_1, q_2 \geq 3$. By \autoref{lem:restr-Qi}, there is a tensor $S \in \bbK^{3} \otimes \bbK^{3} \otimes \bbK^{3}$ such that $T \geq S$ and
    $\subrank_1(S) = \min \{q_1, 3\} = 3$ and $\subrank_2(S) = \min \{q_2, 3\} = 3$. Then $S \degengeq \n^{(3)}$ by \autoref{lem:binding-to-null}.
\end{proof}

\begin{lemma}\label{lem:D}
    Suppose $r_1, r_2, r_3 \geq 3$. Let $i,j,k \in [3]$ be distinct. If $q_i, q_j = 2$, then $q_k = 2$ and $T$ is isomorphic to $\deter$.
\end{lemma}
\begin{proof}
    For $i=1,2,3$, let $\mathcal{A}_i = T^{(i)}(\bbK^{n_i*})$ be the image of the $i$-th flattening of $T$. If any of these is equivalent to the subspace of $3 \times 3$ skew-symmetric matrices then all three are and~$T$ is equivalent to $\deter$. 
    
    Suppose none of these is equivalent to the subspace of $3 \times 3$ skew-symmetric matrices. We will deduce a contradiction. Suppose for simplicity of notation that $q_1 = q_2 = 2$. We apply \autoref{lem:classification} to the space $\calA_1$. Since $r_2, r_3 \geq 3$, $\calA_1$ is not equivalent to a space supported on only two rows or two columns. Therefore, after changing coordinates, we may assume it is supported on the first row and the first column. 
In other words
\[
\calA_1 = \left\{ \left( \begin{array}{cccc} 
\ell_{11} & \ell_{12} & \cdots & \ell_{1r_3} \\
\ell_{21} & & & \\
\vdots & & & \\
\ell_{r_21} & & & 
\end{array}
\right) \in \bbK^{r_2} \otimes \bbK^{r_3} : \ell_{bc} \in \bbK^{r_1}\right\}.
\]
If $\ell'$ is a generic linear combination of $\ell_{11} ,\ell_{21} \vvirg \ell_{r_2 1}$ then $\ell' , \ell_{12} \vvirg \ell_{1r_3}$ are linearly independent, otherwise the rank of the third flattening of $T$ would be smaller than $r_3$. After possibly acting on the second tensor factor of $T$, or equivalently performing row operations on $\calA_1$, assume $\ell_{11} = \ell'$ and applying a linear transformation on the first tensor factor we may assume $\ell_{j1} = e_j$. In this case, the space $\calA_2$ turns out to be
\[
\calA_2 = \left\{  \left( \begin{array}{cccc} e_1 & z_{12} & \cdots & z_{1r_3} \\
& e_1 & & \\
& & \ddots &  \\
& & & e_1 \\
\end{array}\right) \in \bbK^{r_1} \otimes \bbK^{r_3} : z_{1c} \in \bbK^{r_2} \right\}.
\]
If $r_1 \geq 3$, then $\calA_2$ contains a matrix of rank at least $3$, in contradiction with the condition $q_2 \leq 2$. This provides a contradiction and concludes the proof.
\end{proof}

\begin{proof}[Proof of \autoref{thm:new}]
    Suppose $r_i \leq 2$ for some $i \in [3]$. Then we must be in case (a), (b) or (c) by  \autoref{thm:prev}. 
    Suppose $r_i \geq 3$ for all $i \in [3]$. Then we cannot be in case (a), (b) or (c).
    From \autoref{lem:q-lower-bound}, \autoref{lem:N} and \autoref{lem:D} follows that (d) must hold.
\end{proof}

\section{Open problems}

\begin{enumerate}
    \item \autoref{th:new-values} says that the smallest possible values of the asymptotic subrank are $0, 1, 1.88988, 2, 2.68664$. We know that the set of possible values $\geq 2.68664$ is also discrete \cite{briet2023discreteness}. What is the next smallest value?
    As a candidate, we know that there exists a tensor $T$ with $\aQ(T) \approx 2.7551$ \cite[page~132]{strassen1991degeneration}.
    \item What general structure is there in the gaps in the asymptotic subrank? For every natural number $n \in \bbN$ what is the smallest (largest) value that the asymptotic subrank takes that is strictly larger (smaller) than $n$?
\end{enumerate}

\paragraph{Acknowledgements.}
  We thank the organisers of the Workshop on Algebraic Complexity Theory (WACT) 2023 at the University of Warwick, where this project was conceived.
  J.Z.\ was supported by NWO Veni grant VI.Veni.212.284.

\bibliographystyle{alphaurl}
\bibliography{subrank}

\end{document}